\documentclass[a4paper]{article}
 \usepackage{latexsym}
 \usepackage{amsmath}
 \usepackage{enumerate}
 \usepackage{amssymb}

\usepackage{amsthm}
\usepackage{amscd}
\usepackage{esint}

\newcommand{\al}[1]{\left\langle #1 \right\rangle}
 \newcommand{\p}{\varphi}
 \newcommand{\D}{\nabla}
 \newcommand{\R}{\mathbb{R}}
\def\W{\mathcal{W}}

\newtheorem{theorem}{Theorem}[section]
\newtheorem{lemma}[theorem]{Lemma}
\newtheorem{proposition}[theorem]{Proposition}

\newtheorem{remark}[theorem]{Remark}

\DeclareMathOperator{\sym}{sym}
\DeclareMathOperator{\cof}{cof}

\begin{document}
\title{Another remark on constrained von K\'arm\'an theories}

\date{}

\author{
{\sc Peter Hornung}\footnote{
Technische Universit\"at Dresden (Germany).
}
}

\maketitle

\section{Introduction}\label{Sec0}

This note is an appendix to \cite{H-PRSE}.
Let $B\subset\R^2$ denote the open unit disk, let $U_0\in W^{2,2}(B)$, set 
$k = \det\D^2 U_0$ and consider the functional 
$\W_k : W^{2,2}(B)\to [0, \infty]$
given by
$$
\W_k(V) = 
\begin{cases}
\int_B |\D^2 V(x)|^2\ dx &\mbox{ if }\det\D^2 V = k\mbox{ a.e. on }B,
\\
+ \infty &\mbox{ otherwise.}
\end{cases}
$$
The functionals $\W_k$ are scalar variants of the functionals studied
in \cite{H-general}. They are the natural generalisations of the
constrained von K\'arm\'an functionals derived in \cite{fjm2} to the case
when the reference configuration is `prestrained', or to shallow shells,
cf. e.g. in \cite{GemmerVenka}.
In \cite{H-PRSE} the Euler-Lagrange equation
for the functionals $\W_k$ was derived, and it was rigorously justified 
in the elliptic case $k > 0$.
\\
As in \cite{H-PRSE} we say that a function $V\in W^{2,2}(B)$
with $\det\D^2 V = k$ is stationary for $\W_k$ if 
$$
\frac{d}{dt}\Big|_{t = 0} \int_B |\D^2 u(t)|^2 = 0
$$
for all strongly $W^{2,2}$-continuous
maps $t\mapsto u(t)$ from a neighbourhood of zero in $\R$ into $\{\W_k < \infty\}$
such that $u(0) = V$ and such that the derivative $u'(0)$ exists (at least weakly in $W^{2,2}$).
A function $V\in W^{2,2}(B)$ with $\det\D^2 V = k$ is said to be formally stationary for $\W_k$
if 
$$
\int_B \D^2 V : \D^2 f = 0 \mbox{ for all }f\in W^{2,2}(B)\mbox{ with }\cof\D^2 V : \D^2 f = 0\mbox{ a.e. in }B.
$$
Clearly, formally stationary functions are stationary.
\\
In this note we are interested in radially symmetric functions.
Theorem 1.7 in \cite{LeOcPa} asserts that if $V\in W^{2,2}(B)$ is radially symmetric
with $k := \det\D^2 V \geq c$ for some positive constant $c$ and $k\in C^{\infty}(\overline B)$, 
then $V$ is a stationary point
of $\W_k$ in the sense of \cite{H-PRSE}. The approach in \cite{LeOcPa}
is based on the pointwise Euler-Lagrange equation from \cite{H-PRSE}.
\\
Here we obtain the same conclusion without any hypotheses on $k$. 
The proof is very simple and robust, and it makes no explicit
use of the (pointwise) Euler-Lagrange equation from \cite{H-PRSE},
but of course it also uses the concept of stationarity introduced there.
\begin{proposition}\label{main} 
Let $V\in W^{2,2}(B)$ be radially symmetric and let $k = \det\D^2 V$. Then $V$ 
is a stationary point of $\W_k$.
\end{proposition}

{\bf Remarks.}
\begin{enumerate}[(i)]
\item In fact we obtain that $V$ is even formally stationary.
\item In the convex case $k \geq c > 0$
the equation $\det\D^2 V = k$ admits at most one radially symmetric solution,
cf. \cite[Corollary 6.2]{LeOcPa}.
\\
In contrast to that, in Section \ref{Example} we give examples of
$k\in C^{\infty}(\overline B)$ such that there exist infinitely many radially symmetric solutions 
to $\det\D^2 V = k$. By Proposition \ref{main} each of them is stationary for $\W_k$.
\item By Remark \ref{same} below all radially symmetric solutions $V\in W^{2,2}(B)$ 
of $\det\D^2 V = k$ have the same energy $\W_k(V)$.
\end{enumerate}

\section{Proof of Proposition \ref{main}}

\begin{lemma}\label{vks} 
Let $V\in W^{2,2}(B)$ be radially symmetric and assume that
$\int_S \D^2 V : \D^2 f = 0$
for all radially symmetric $f\in W^{2,2}(B)$ with $\cof\D^2 V : \D^2 f = 0$.
Then in fact $\int_B \D^2 V : \D^2 f = 0$ for {\em all} 
$f\in W^{2,2}(B)$ with $\cof\D^2 V : \D^2 f = 0.$
\end{lemma}

\begin{proof}
For $\p\in [0, 2\pi]$ denote by $e^{i\p}\in SO(2)$ the counter-clockwise rotation by the angle 
$\p$ and set $\rho_{\p}(x) = e^{i\p}x$, and for $F : B\to \R^{2\times 2}_{\sym}$
set $\rho_{\p}^*F = e^{-i\p}(F\circ \rho_{\p})e^{i\p}$. Clearly 
$\D^2 (f\circ\rho_{\p}) = \rho_{\p}^*(\D^2 f)$ for any $f\in W^{2,2}(B)$.
\\
For arbitrary $f\in W^{2,2}(B)$ define $\overline f : B\to\R$ by setting
$$
(\overline f)(x) = \int_0^{2\pi} f(e^{i\p}x)\ \frac{d\p}{2\pi}.
$$
and for $F\in W^{2,2}(B, \R^{2\times 2}_{\sym})$ set
$$
\al{F} = \int_0^{2\pi} (\rho_{\p}^*F)\ \frac{d\p}{2\pi}. 
$$
For $f\in (C^2\cap W^{2,2})(B)$, it is easy to see that
\begin{equation}
\label{vks-0} 
\D^2 \overline f = \al{\D^2 f}\mbox{ almost everywhere on }B.
\end{equation} 
But
clearly $f_n\to f$ in $L^2(B)$ implies $\overline f_n\to \overline f$ in $L^2(B)$,
and similarly $F_n\to F$ in $L^2(B, \R^{2\times 2}_{\sym})$
implies $\al{F_n}\to\al{F}$ in $L^2(B, \R^{2\times 2}_{\sym})$.
Hence \eqref{vks-0} remains true for all $f\in W^{2,2}(B)$.
\\
We must show that, for $V$, $f\in W^{2,2}(B)$ with $V$ radially symmetric,
\begin{equation}
\label{vks-1}
\int_{B} \D^2 V : \D^2 f = \int_{B} \D^2 V : \D^2\overline f
\end{equation} 
and
\begin{equation}
\label{vks-2}
\cof\D^2 V : \D^2 \overline f = \overline{\cof \D^2 V : \D^2 f}.
\end{equation} 
In fact, take these equalities for granted, and assume that $V$ satisfies the hypotheses
of the lemma.
Let $f\in W^{2,2}(B)$ be a (possibly non-symmetric) solution of $\cof\D^2 V : \D^2 f = 0$.
Then \eqref{vks-2} implies that $\cof\D^2 V : \D^2 \overline f = 0$. Hence the hypotheses on $V$ imply 
that $\int_{B} \D^2 V:\D^2 \overline f = 0$. Hence the claim follows from \eqref{vks-1}.

To prove \eqref{vks-1} and \eqref{vks-2} let us show that
\begin{equation}
\label{vks-3}
F : \D^2\overline f = \overline{F : \D^2 f}\mbox{ for all }f\in W^{2,2}(B),
\end{equation} 
whenever $F\in L^2(B, \R^{2\times 2}_{\sym})$ satisfies $F = \rho_{\p}^*F$ for all $\p$.
\\
In fact, since $\rho_{\p}^*F = F$ we compute
$$
(F : \D^2 f)\circ\rho_{\p} = (\rho_{\p}^*F):(\rho_{\p}^*\D^2 f) = F:(\rho_{\p}^*\D^2 f).
$$
Integration over $\p$ yields $\overline{F : \D^2 f} = F : \al{\D^2 f}$, which by \eqref{vks-0}
is just \eqref{vks-3}.
\\
Equation \eqref{vks-2} is \eqref{vks-3} with $F = \cof\D^2 V$. Applying \eqref{vks-3}
with $F = \D^2 V$, integrating over $B$ and applying Fubini we find
$$
\int_{B} \D^2 V : \D^2\overline f 
= \int_0^{2\pi} \left( \int_{B} (\D^2 V : \D^2 f )(\rho_{\p}(x))\ dx \right)\ \frac{d\p}{2\pi}.
$$
Now \eqref{vks-1} follows by a change of variables, which renders the inner integral $\p$-independent.
\end{proof}

\begin{lemma}\label{le-radi} 
Let $V$, $f\in W^{2,2}(B)$ be radially symmetric, and assume that $\cof\D^2 V : \D^2 f = 0$
almost everywhere. Then $\D^2 V : \D^2 f = 0$ almost everywhere.
\end{lemma}
\begin{proof}
Abusing notation we write $V(x) = V(|x|)$ and $f(x) = f(|x|)$. A short computation shows that
\begin{equation}
\label{cofrad} 
\cof\D^2 V(x) : \D^2 f(x) = V''(|x|) \frac{f'(|x|)}{|x|} + 
f''(|x|) \frac{V'(|x|)}{|x|}\mbox{ for a.e. }x\in B.
\end{equation} 
Hence $V'f'$ is constant on $(0, 1)$. Since $V\in W^{2,2}(B)$,
we clearly have $V'(0) = 0$ (this follows e.g. from the formula \eqref{same-4}).
Hence
\begin{equation}
\label{thevo-1}
V' f' = 0 \mbox{ almost everywhere on }(0, 1).
\end{equation} 
Hence $f' = 0$ almost everywhere on $\{V'\neq 0\}$. In particular,
$f'' = 0$ almost everywhere on $\{V'\neq 0\}$. On the other hand,
$V'' = 0$ almost everywhere on $\{V' = 0\}$. Hence
$
V'' f'' = 0 \mbox{ almost everywhere on }(0, 1).
$
So the claim follows from the equality
\begin{equation}
\label{tra-22} 
\D^2 V(x):\D^2 f(x) = V''(|x|) f''(|x|) + \frac{V'(|x|) f'(|x|)}{|x|^2} \mbox{ for a.e. }x\in B.
\end{equation} 
\end{proof}

\begin{proof}[Proof of Proposition \ref{main}]
By Lemma \ref{vks} it is enough to show that $\int_B \D^2 V : \D^2 f = 0$ for all radially
symmetric $f\in W^{2,2}(B)$ satisfying $\cof\D^2 V : \D^2 f = 0$. By Lemma \ref{le-radi}
this is indeed the case.
\end{proof}

\section{Further remarks}\label{Example} 

We begin with a trivial observation.

\begin{lemma}
\label{samele} 
Let $U$, $V\in W^{2,2}(B)$ be radially symmetric and write $U(x) = u(|x|)$ 
and $V(x) = v(|x|)$. Then we have
$$
\det\D^2 U = \det\D^2 V \mbox{ a.e. on }B\iff |u'| = |v'| \mbox{ a.e. on }(0, 1).
$$
\end{lemma}
\begin{proof}
By \eqref{cofrad} we have
$
\det\D^2 V(x) = (2|x|)^{-1}((v')^2)'(|x|).
$
But to have $V\in W^{2,2}(B)$ we must have $v'(0) = 0$,
and the same is true for $u$.
\end{proof}

\begin{remark}\label{same}
If $U$, $V\in W^{2,2}(B)$ are radially symmetric and satisfy $\det\D^2 U = \det\D^2 V$
almost everywhere, then $|\D^2 U|^2 = |\D^2 V|^2$ almost everywhere.
In particular, $\W_k(U) = \W_k(V)$ for $k = \det\D^2 V$ (and indeed for any $k$).
\end{remark}
\begin{proof}
We write $V(x) = v(|x|)$ and $U(x) = u(|x|)$. By Lemma \ref{samele} we have
\begin{equation}
\label{same-1} 
(v')^2 = (u')^2 \mbox{ almost everywhere.}
\end{equation} 
Hence $v'' v' = u'' u'$ by the Leibniz rule, and in particular
$
(v'')^2 (v')^2 = (u'')^2 (u')^2.
$
Hence
$$
(v'')^2 = (u'')^2 \mbox{ almost everywhere on }\{v'\neq 0\} = \{u' \neq 0\};
$$
the (almost everywhere) equality of the sets follows from \eqref{same-1}.
But almost everywhere on $\{v' = 0\}$ we have $v'' = 0$, and similarly for $u$.
Summarising, we have
$
|v'| = |u'|$
and 
$|v''| = |u''|$ almost everywhere.
The claim follows by combining this with the formula
\begin{equation}
\label{same-4} 
|\D^2 V(x)|^2 
= (v''(|x|))^2 + |x|^{-2}(v'(|x|))^2,
\end{equation}
which follows by taking $V = f = v$ in \eqref{tra-22}.
\end{proof}

\subsection*{An extreme example}

For a given radially symmetric $V$, Lemma \ref{samele} provides a characterisation of 
all radially symmetric
solutions $U$ of $\det\D^2 U = \det\D^2 V$. Obviously, for generic $v$ 
there can be multiple solutions
to the equation $|u'| = |v'|$. We only give one class of examples; it is easy to construct
many others. 
\\
Let $\eta\in C^{\infty}(\R)$ be nonnegative and supported in 
$(-\frac{1}{2}, \frac{1}{2})$ (but not identically zero).
Let $R\in (0, 1]$ and let
$(t_n)_{n = 1}^{\infty}\subset (0, R)$ be a strictly increasing sequence with 
$t_n\uparrow R$. Set $t_0 = 0$ and define $v:[0,1) \to \R$ by setting
$$
v(t) = \sum_{n = 0}^{\infty} (t_{n+1} - t_n)^n\ \eta\left( \frac{2t - t_n - t_{n+1}}{2|t_{n+1} - t_n|}\right).
$$
Since $v\in W^{2,\infty}(0, 1)$ vanishes near zero, we see that $V(x) = v(|x|)$
satisfies $V\in W^{2,\infty}(B)$; in fact we have $V\in C^{\infty}(\overline B)$, 
because $v\in C^{\infty}([0, 1])$.
For each $n = 1, 2, 3, ...$ define $u_n : [0, 1)\to\R$ by
$$
u_n(t) = 
\begin{cases}
-v(t) &\mbox{ if }t\in (t_n, t_{n+1})
\\
v(t) &\mbox{ otherwise,}
\end{cases}
$$
and define $U_n\in W^{2,\infty}(B)$ by setting $U_n(x) = u_n(|x|)$.
Since $|u_n'| = |v'|$ almost everywhere, Lemma \ref{samele} shows that $\det\D^2 U_n = \det\D^2 V$.
Clearly, the $U_n$ are pairwise distinct and radially symmetric. So 
by Proposition \ref{main} the functional
$\W_k$ with $k = \det\D^2 V$ admits infinitely many stationary points.

\def\cprime{$'$}

\end{document}